\documentclass[a4paper,11pt]{article}
\usepackage{amsfonts}
\usepackage{amssymb}
\usepackage{latexsym}
\usepackage{amsmath}
\usepackage{amscd}
\usepackage{amsthm}
\usepackage{hyperref}
\usepackage{url}

\newtheorem{theorem}{Theorem}[section]
\newtheorem{definition}{Definition}

\newtheorem{proposition}[theorem]{Proposition}

\newtheorem{lemma}[theorem]{Lemma}

\newcommand{\N}{{\mathbb N}}
\newcommand{\Z}{{\mathbb Z}}
\renewcommand{\P}{{\mathbb P}}

\newcommand{\x}{{\mathbf x}}
\newcommand{\y}{{\mathbf y}}

\newcommand{\w}{{\mathbf w}}
\newcommand{\z}{{\mathbf z}}

\newcommand{\footremember}[2]{%
\footnote{#2}
\newcounter{#1}
\setcounter{#1}{\value{footnote}}%
}

\title{Existence and coalescence of directed infinite geodesics
in the percolation cone for Durrett-Liggett class of measures}
\author{%
  Kumarjit Saha\footremember{TIFR}{TIFR Center for Applicable Mathematics, Bangalore.}
  \footnote{This work is benefited from the support of the AIRBUS Group Corporate Foundation Chair in Mathematics of Complex Systems established in TIFR CAM.}}

\date{\today}
\begin{document}

\maketitle

\begin{abstract}
For first passage percolation (FPP) on integer lattice 
 with i.i.d. passage time distributions, in order to show existence of semi-infinite geodesics 
 along a fixed direction, one requires unproven assumptions on the limiting shape.
We consider FPP on two-dimensional integer lattice with i.i.d. passage times distributed as Durrett-Liggett class of measures. For this model, we show that  along any direction
in a deterministic angular sector (known as percolation cone), starting from every lattice point there exists an infinite geodesic along that direction and such directed geodesics coalesce almost surely. 
We prove that for this  model, bi-infinite geodesics exist almost surely.
Our proof does not require any assumption on the limiting shape. 
 \end{abstract}


\section{Introduction}
First passage percolation (FPP) was introduced in 1965 by Hammersley and Welsh \cite{HW65} as a  stochastic model for fluid flow through a porous medium. We consider FPP on the $2$-dimensional integer lattice $(\mathbb{Z}^2, {\cal E}^2)$. Two vertices $\x, \y \in \mathbb{Z}^2$ are said to be neighbour if $||\x - \y||_1 = 1$  and the edge set ${\cal E}^2$ consists of all the line segments between the neighbouring vertices. To each edge $e \in {\cal E}^2$, a strictly positive random passage time $t(e)$ is attached. Let ${\mathbb P}$ denote the corresponding probability measure.
For any two lattice points $\x = (\x(1),\x(2))$  and $\y = (\y(1),\y(2))$ in $\mathbb{Z}^2$, the passage time between $\x$ and $\y$ is denoted by 
\begin{align*}
\label{eq:PassageTimeSites}
\tau(\x, \y) := \inf_{\gamma:\x \leftrightarrow \y}\tau(\gamma),
\end{align*}
where the infimum is taken over all finite lattice paths from $\x$ to $\y$. For 
any finite lattice path $\gamma$, the passage time $\tau(\gamma)$ across the path $\gamma$ is the sum of the passage times $t(e)$
attached to the edges along the path $\gamma$. 
A geodesic from $\x$ to $\y$ is a finite lattice path $\gamma$
from $\x$ to $\y$ such that $\tau(\gamma) = \tau(\x,\y)$. Here, with a slight abuse of terminology,  
 a `lattice path' or a `path' $\gamma$ means  a  sequence of neighbouring lattice points 
 $\{\x_i :i \in I, \x_i \in \mathbb{Z}^2 \}$,
  where the index set $I$ can be any (finite or infinite) subset of $\mathbb{Z}$.
For $i \in I$, we denote the vertex $\x_i$ as $\gamma(i)$. For a semi-infinite path $\gamma$
we usually take $I= \mathbb{N}$, whereas for a bi-infinite path $\gamma$
we take $I= \mathbb{Z}$. 
An \textit{infinite geodesic} is an infinite lattice path such that each finite 
subpath is a geodesic. 

In the mid $90$'s people started
working on the understanding of infinite 
geodesics for lattice FPP (see \cite{LN96}, \cite{WW98}). This was partly motivated by the connection between the existence of bi-infinite geodesics in FPP and the existence of non-constant ground state
in the disordered ferromagnetic  spin models (see \cite{N94} for more details).
The main questions involve (i) existence and uniqueness of semi-infinite geodesic along a fixed direction,
(ii) coalescence of such directed geodesics (starting from all the lattice points) and 
(iii) absence of bi-infinite geodesics.
In general for $\theta \in [0,2\pi)$, a semi-infinite 
path $\gamma$ is said to be along the direction $\theta$ if and only if
\begin{equation}
\label{eq:GeodesicDirection}
\lim_{n \to \infty} \text{arg}\bigl(\gamma(n)\bigl /||\gamma(n)||_1\bigr)=\theta,
\end{equation}
where for any unit vector $\x$, $\text{arg}(\x)$ denotes the argument of $\x$. 
In what follows, a semi-infinite path is simply referred as 
an infinite path.   
Two infinite lattice paths $\gamma_1$ and $\gamma_2$
are said to coalesce if the tails of the corresponding sequences of neighbouring vertices eventually become same, i.e., there exist 
$i_1, i_2 \in \mathbb{N}$ such that $\gamma_1(i_1 + k) = \gamma_2(i_2 + k)$ for all $k\geq 0$. 

Till now, for FPP on integer lattice with i.i.d. passage time distributions, 
these questions are far from completely understood. The main theorems proved to date require unproven assumptions on the limiting shape.
In order to state these assumptions, we need to describe the associated  limiting shape. 

We extend the notion of passage times for general $\x, \y \in \mathbb{R}^2$, 
by setting $\tau(\x, \y)= 
\tau(\tilde{\x}, \tilde{\y})$ where $\tilde{\x}, \tilde{\y} \in \mathbb{Z}^2$ are the unique 
lattice points such that $\x \in \tilde{\x} + [-1/2, 1/2)^2$ and 
$\y \in \tilde{\y} + [-1/2, 1/2)^2$. For $t > 0$ set $B(t) := \{\x \in \mathbb{R}^2 : \tau(\mathbf{0}, \x) \leq t\}$ and $B(t)/t := \{\x/t : \x \in B(t)\}$. Under mild conditions on passage time distributions it is shown that (\cite{B90}, \cite{CD81})
there exists a deterministic, compact convex set ${\cal B}$,
called as the limiting shape, with non empty interior such that for all $\epsilon > 0$
\begin{align}
{\mathbb P}((1-\epsilon){\cal B} \subseteq B(t)/t \subseteq (1+\epsilon){\cal B}
\text{ for all large }t) = 1.
\end{align}

Though it is believed that for any product measure on $\mathbb{Z}^2$ with continuous passage time distributions, the limiting shape should have ``uniformly positive curvature" (for 
a definition of uniformly positive curvature see \cite{N94}), there is not a single proven example till date.
Under the assumptions that 
\begin{itemize}
\item[($A_1$)] ${\mathbb P}$ is a product measure with continuous marginals having
 exponential moments,
\item[($A_2$)] the limiting shape $\mathcal{B}$ has uniformly positive curvature,
\end{itemize}
Newman showed the following (Theorem 2.1 of \cite{N94}):
\begin{theorem}
\label{th:NewmanLiceaFPPd2}
Let $\nu$ be any continuous (Borel) probability measure on $[0,2\pi)$. Then $\nu$
almost surely for any $\theta \in [0,2\pi)$, 
there exists a unique infinite geodesic starting from the origin with asymptotic direction $\theta$.
\end{theorem}

We note that, in order to show existence of infinite geodesic even in one direction,
the above theorem requires a global curvature assumption, establishment of which remains 
a major challenge. On the other hand, 
Haggstrom and Meester \cite{HM95} proved that, any compact convex set symmetric about the origin  with non-empty interior, can be realized as
the limiting shape for an FPP model associated to stationary
passage times.

Later Damron and Hanson \cite{DH14} weaken these assumptions by
replacing the global curvature assumption with a directional derivative assumption
and able to prove similar results. 
In order to describe their assumption we introduce some notation.
For $\theta \in [0, 2\pi)$ let $v_\theta$ be the unique point on $\delta({\cal B})$, 
the boundary of ${\cal B}$, with argument $\theta$. A supporting line $L$ for ${\cal B}$ at $v_\theta$ is a line 
that touches ${\cal B}$ at $v_\theta$ and ${\cal B}$ stays on one side of $L$. Under the assumption that
$\delta({\cal B})$ is differentiable at $v_\theta$,
 the supporting line $L$ for ${\cal B}$ at $v_\theta$ is unique.
$I_\theta$, the interval of angles, is defined   as $I_\theta := \{\theta^\prime : v_{\theta^\prime} \in L \}$.
Following the notation of \cite{DH14}, the point $v_\theta$ is called 
an \textit{exposed point of differentiability} if $I_\theta  = \{v_\theta\}$.

For i.i.d. passage time distributions, under a finite first moment assumption, Damron and Hanson proved the following (Theorem 1.1 of \cite{DH14}):
%
%
\begin{theorem}
\label{thm:DHExistenceDirectedGeodesic}
If $\delta({\cal B})$ is differentiable at $v_\theta$, then almost surely starting from the origin, there exists an infinite geodesic asymptotically directed in $I_\theta$. 
\end{theorem}
This means that the limit in (\ref{eq:GeodesicDirection}) exists and it belongs to $I_\theta$.
According to the above theorem, in order to have infinite geodesics along 
the direction $\theta$, the point $v_\theta$ has to be an exposed point of differentiability. It is natural to ask whether there exists any i.i.d. passage time distributions whose limiting shape has exposed point(s) of differentiability on the boundary. Below we describe the only proven example as given in \cite{DH14}:
 
On the lattice $(\mathbb{Z}^2, {\cal E}^2)$, let ${\mathbb P}_p$ be a product probability measure of i.i.d. passage times with the common distribution function $F$ satisfying 
\begin{align}
\label{def:DurettLiggettClass}
\inf \text{supp} F = 1 \text{ and } F(1) = p \geq p^{\rightarrow}_{c},
\end{align} 
where $p^{\rightarrow}_{c}$ is the critical value for oriented percolation on $\mathbb{Z}^2$.
These are called Durrett-Liggett class of measures. For this class of measures, Durrett and
Liggett \cite{DL81} observed that the limiting shape
has a `flat edge' which was later completely characterized by Marchand \cite{M05}. 
For $p > p^{\rightarrow}_{c}$, let $\alpha_p (>0)$ denote the asymptotic speed of 
the rightmost infinite open path for a supercritical oriented percolation with parameter $p$ (see \cite{D84}). Consider the two points $M_p, N_p \in \mathbb{R}^2$ where $M_p := ( 1/2 + \alpha_p/\sqrt{2} , 1/2 - \alpha_p/\sqrt{2} )$ 
and $N_p := ( 1/2 - \alpha_p/\sqrt{2} , 1/2 + \alpha_p/\sqrt{2} )$.
The line segment obtained by joining these two points is denoted by $[M_p, N_p]$. 
The following theorem due to Marchand
characterizes the `flat edge' completely (Theorem 1.3 of \cite{M05}).

\begin{theorem}
\label{thm:FlatEdge}
Let $F$ be a distribution on $\mathbb{R}^+$ satisfying (\ref{def:DurettLiggettClass}) and let ${\cal B}$ be the corresponding limit shape.
\begin{enumerate}
\item ${\cal B} \subseteq \{\x \in \mathbb{R}^2 : ||\x||_1 \leq 1\}$.
\item If $p > p^{\rightarrow}_{c}$, then ${\cal B} \cap \{\x \in \mathbb{R}_+ \times \mathbb{R}_+ : ||\x||_1 = 1\}
 = [M_p , N_p ]$ and the segment
$[M_p , N_p ]$ is called the flat edge of the limiting shape ${\cal B}$.

\item If $p = p^{\rightarrow}_{c}$, then ${\cal B} \cap \{\x \in \mathbb{R}_+ \times \mathbb{R}_+ : ||\x||_1 = 1\} = {(1/2, 1/2)}$.
\end{enumerate}
\end{theorem}

 Auffinger and Damron \cite{AD13} proved that for $p > p^{\rightarrow}_{c}$, the boundary $\delta({\cal B})$ is differentiable at both the points $M_p$ and $N_p$ whereas for $p = p^{\rightarrow}_{c}$, it is differentiable at $(1/2,1/2)$. 
This, together with the above theorem, prove that the point $(1/2,1/2)$  
   is an exposed point of differentiability for i.i.d. passage times satisfying (\ref{def:DurettLiggettClass}) with $p = p^{\rightarrow}_{c}$. For FPP with i.i.d. passage time distributions, this is the only proven example for having an 
    exposed point of differentiability on the boundary of the limiting shape.
    Let $\theta^-_p, \theta^+_p$ be the angles made by the points $M_p$ and $N_p$ at the origin respectively. The angular sector $(\theta^-_p, \theta^+_p) = \bigl(\tanh(\frac{ 1/2 - \alpha_p/\sqrt{2} }{  1/2 + \alpha_p/\sqrt{2} }), \tanh(\frac{ 1/2 - \alpha_p/\sqrt{2} }{  1/2 + \alpha_p/\sqrt{2} })\bigr)$ known as `percolation cone'.
Theorem \ref{thm:FlatEdge}, also gives us that
 for $p > p^{\rightarrow}_{c}$ there is \textit{no} exposed point of differentiability
$v_\theta$ with argument $\theta$ for $\theta \in (\theta^-_p, \theta^+_p)$. In this case, clearly the limit shape does not have uniform positive curvature. We mention here that
 for $p > p^{\rightarrow}_{c}$, differentiability at both the points $M_p$ and $N_p$ 
proves that the limiting shape is non-polygonal.

The above discussion shows that for i.i.d. passage times satisfying (\ref{def:DurettLiggettClass}), Theorem \ref{thm:DHExistenceDirectedGeodesic} gives us the following:
\begin{itemize}
\item[(i)] for $p= p^\rightarrow_{c}$, there exists 
an infinite geodesic starting from the origin along the direction $\pi/4$ and
\item[(ii)] for $p > p^\rightarrow_{c}$, there exists 
an infinite geodesic starting from the origin asymptotically 
directed in $[\theta^-_p, \theta^+_p]$.
\end{itemize}
For simplicity of notation, in both (i) and (ii) and in what follows, we talk about 
existence of infinite geodesics in the positive quadrant only. Since the limiting shape 
is symmetric with respect to the co-ordinate axes,
similar results hold for the other three quadrants also.
 
We consider  FPP with i.i.d. passage time distributions satisfying (\ref{def:DurettLiggettClass}) with $p > p^\rightarrow_{c}$. One of the main results of this paper is that, for any direction within the percolation cone, starting from each lattice point there exists an infinite geodesic along that direction.
\begin{theorem}
\label{thm:DirectedGeodesicsford2}
Let $F$ be a distribution on $\mathbb{R}^+$ satisfying (\ref{def:DurettLiggettClass}) with $p>p^{\rightarrow}_{c}$.
Then for any $\theta \in [\theta^-_p, \theta^+_p]$,
there exists an infinite geodesic starting from the origin with asymptotic direction $\theta$ almost surely.
\end{theorem}
As discussed earlier, this result does not follow from the 
earlier works. It is useful to mention that we have no assumptions on the limiting 
shape and we do not require any moment assumptions for the passage time distribution.
Since we are working with atomic passage time distributions, 
these directed geodesics need not be unique.

The next set of results deal with coalescence of infinite directed geodesics. 
Before stating our result about coalescence we discuss 
some previous results. For FPP on $\mathbb{Z}^2$ under the assumptions ($A_1$) and ($A_2$),
 Licea and Newman \cite{LN96} proved that for any $\theta \in [0,2\pi)$, all $\theta$-directional geodesics almost surely coalesce except some deterministic Lebesgue null set $D$
 of $[0,2\pi)$. Since then it has been an open problem to show that the set $D$ can be taken to be empty. Zerner \cite{N97} proved that the set $D$ is at most countable. Later Damron and Hanson further improved it by showing the following (Theorem of \cite{DH14}):
\begin{theorem}
\label{thm:CoalGeoDamronHanson}
If $v_\theta \in \delta({\cal B})$ is an exposed point of differentiability, then almost surely
there exists a collection $\{\gamma_{\x} : \x \in \mathbb{Z}^2\}$ of infinite geodesics such that  the following are true:
\begin{enumerate}
 \item $\gamma_\x$ starts from $\x$ and has asymptotic direction $\theta$.
 \item For all $\x, \y \in \mathbb{Z}^2$, the infinite geodesics $\gamma_\x$ and $\gamma_\y$ coalesce.
 \item For each $\x \in \mathbb{Z}^2$ the set $\{\y : \x \in \gamma_\y\}$ is finite.
\end{enumerate}
\end{theorem}
Recently based on differentiability assumption only, Ahlberg and Hoffman \cite{AH16} improved the results of \cite{DH14} partly. 
Now we describe our result for Durrett-Liggett class of measures.
\begin{theorem}
\label{thm:CoalGeoesic}
Let $F$ be a distribution on $\mathbb{R}^+$ satisfying (\ref{def:DurettLiggettClass}) with $p > p^{\rightarrow}_{c}$. Then almost surely for any $\theta \in [\theta^-_p, \theta^+_p]$,
there exists a collection $\{\gamma_{\x} : \x \in \mathbb{Z}^2\}$ of infinite geodesics such that  the following are true:
\begin{enumerate}
 \item $\gamma_\x$ starts from $\x$ and has asymptotic direction $\theta$.
 \item For all $\x, \y \in \mathbb{Z}^2$, the infinite geodesics $\gamma_\x$ and $\gamma_\y$ coalesce.
\end{enumerate}
\end{theorem}
It is important to mention here that since the directed geodesics are not unique, 
the above theorem does not imply that all $\theta$-directional geodesics coalesce. For atomic passage time distribution, same comment holds true for Theorem \ref{thm:CoalGeoDamronHanson} also. We again mention that these results do not require any moment assumptions or assumptions about the limiting shape.

The last result of Theorem \ref{thm:CoalGeoDamronHanson} is related to existence of bi-infinite geodesics, a question apparently first
posed for lattice FPP by H. Furstenberg (see page 258 of \cite{K86}).
 Generally it is believed that bi-infinite geodesics do not exist.
 Till date, the best result known for lattice FPP with i.i.d. continuous passage times, is due to Damron and Hanson \cite{DH16} who proved that there is no bi-infinite geodesics with one end having a deterministic direction. Our next theorem shows that for Durrett-Liggett class of measures, bi-infinite geodesics do exist almost surely.
\begin{theorem}
\label{thm:BiInfGeoesic}
Let $F$ be a distribution on $\mathbb{R}^+$ satisfying (\ref{def:DurettLiggettClass}) with $p > p^{\rightarrow}_{c}$. Then bi-infinite geodesic exists almost surely.
\end{theorem}
In this regard it is worth pointing out that, recently Benjamini and Tessera \cite{BT16}
showed that bigeodesics may exist on graph with different geometry.

Since we are working with i.i.d. passage time distributions satisfying (\ref{def:DurettLiggettClass})  with $p>p^{\rightarrow}_{c}$, a natural copuling 
with supercritical oriented percolation exists (for more details 
see the beginning of next section) and 
every `oriented infinite open' path gives an infinite geodesic.
On the event that the origin is a percolation point, Theorem \ref{thm:DirectedGeodesicsford2} gives that for any direction in the percolation cone, there
exists an infinite oriented open path directed along that direction.
On the complement event, i.e., when origin is not a percolation point, we use bi-directional percolation points (see (\ref{def:BiPercPoint})) to create a bounded region around the origin
enclosed by geodesic paths. This construction followed by a non-crossing argument complete the proof. This non-crossing argument 
depends on planar structure of $\mathbb{Z}^2$. 

We mentioned earlier that for supercritical oriented percolation, using subadditive ergodic theorem, Durrett \cite{D84} proved that the rightmost infinite oriented open path starting from a percolation point has a deterministic direction almost surely.
Though it is possible to modify the arguments of \cite{D84} to obtain infinite oriented open path along any direction in the percolation cone,
we need to show existence of directed infinite geodesics for non-percolation points as well.  
In this paper, following \cite{BCDN13}, we provide a constructive proof for any direction inside the percolation cone. This involves a local construction to approximate an infinite oriented open path giving suitable stopping times.
These stopping times allow us in obtaining a Markovian structure through regenerations.
This Markovian structure together with the bi-directional percolation points, enable us to apply Lyapunov function technique to create a region enclosed by geodesic paths. 
This modification is required to work with non-percolation points. 

\section{Proof of the theorems}

We first couple our FPP model with an embedded oriented percolation model.
 For Durrett-Liggett class of measures, Marchand \cite{M05} first observed and used this coupling.
 
Let us first introduce the embedded oriented percolation model and some notation. 
An edge $e \in {\cal E}^2$ is said to be open if $t(e) = 1$, and closed otherwise. 
Each vertex $\x \in \mathbb{Z}^2$ has $2$ oriented edges and when these edges
 are open, they give access
to the vertices $\x + (1, 0)$ and $\x +  (0,1)$.
An oriented open path is a path consisting of oriented open edges only. The event 
$\{ \x \leadsto \y\}$ denotes that there is an oriented open path from $\x$ to $\y$. 
For $\x \in \mathbb{Z}^2$, let $B(\x)$ denote the event that $\x$ 
has an \textit{infinite} oriented open path. Clearly $\mathbb{P}_p(B(\x))$ does not depend on the vertex
$\x$ and it is well known that (see \cite{D84}, \cite{BG90})
\begin{align*}
\mathbb{P}_p(B(\mathbf{0})) = 
\begin{cases}
= 0 &\text{ for }p \leq p^\rightarrow_{c}\\
> 0 &\text{ for }p > p^\rightarrow_{c}.
\end{cases}
\end{align*}
 
For $\mathbb{P}_p(\cdot)$, as described in (\ref{def:DurettLiggettClass})
with $p > p^\rightarrow_{c}$, we prove the following proposition.
\begin{proposition}
\label{prop:DirectedPathforOPd2}
Conditional on the event $B(\mathbf{0})$, for any $\theta \in [\theta^-_p, \theta^+_p]$
 there exists an infinite oriented open path $\gamma$ starting from the origin along the direction $\theta$ almost surely. 
\end{proposition}
This proposition will be proved through a sequence of lemmas.
We describe a ``local" construction to obtain
an infinite oriented open path $\gamma$.
 This construction is motivated from \cite{BCDN13} to run 
a symmetric random walk on the backbone of a supercritical oriented site percolation cluster.
We did necessary modifications to work with supercritical oriented bond percolation
and for $q \in [0,1]$, we run a `$q$-walk' with drift which depends on the probability $q$.
 Similar construction in more restrictive setting
was used in \cite{SS13}
 to analyse the rightmost infinite open paths.   
We mostly follow the notation used in \cite{BCDN13}.

Let $\{U_{(x,t)} : (x,t) \in \mathbb{Z}^2\}$ be a collection of i.i.d. $U(0,1)$
 random variables, independent of the collection $\{t(e) : e \in {\cal E}^2\}$ of i.i.d. passage times, that 
 we have started with. This collection helps us to create 
an additional randomness which is required.

 For $(x,t) \in \mathbb{Z}^2$, let $V(x,t) := \{(x+1,t),(x,t+1)\}$
be the set of the oriented neighbours of $(x,t)$.
For every $(x,t) \in \mathbb{Z}^2$,                                                                                                                                                                                                                                                                                                                                                                                                                                                                                                                                                                                                                                                                                                                                                                                                                                                                                                                                                                                                                                                                                                                                               let 
\begin{align*}
l(x,t) = l_\infty(x,t) := \sup\{ & k \geq 0 : (x,t) \leadsto (y,s) \text{ for some }(y,s) \in \mathbb{Z}^2 \text{ with }y \geq x, s \geq t\\
& \text{ and }||(x,t)-(y,s)||_1 = k\},
\end{align*}
be the length of the 
longest oriented open path starting at $(x,t)$. By definition, if both the oriented edges starting at $(x,t)$ are closed then $l(x,t) = 0$. 
For every $k\in \mathbb{N} $ let 
$l_k(x,t) := l(x,t)\wedge k$ be the length of the longest oriented open path
of length at most $k$ starting at $(x,t)$.
For $k \geq 1$ let 
\begin{align*}
M_k(x,t) := \{(y,s) \in V(x,t) : l_{k - 1}(y,s) + 1 = l_k(x,t)\text{ and } (x,t) \leadsto (y,s)\}.
\end{align*}
If all the oriented edges starting at $(x,t)$ are closed then 
$M_k(x,t) = \emptyset$ for all $k\geq 1$. 

Fix $q \in [0,1]$. 
On the event $\{M_k(x,t) \neq \emptyset \}$,
we define $m^q_k(x,t) \in M_k(x,t)$ as,
\begin{align*}
m^q_k(x,t) :=
\begin{cases}
(x+1,t) & \text{ if }M_k(x,t) = \{(x+1,t)\}\\
& \text{ or if }\# M_k(x,t) = 2 \text{ and }U_{(x,t)}\leq q\\
(x,t+1) & \text{ if }M_k(x,t) = \{(x,t+1)\}\\
& \text{ or if }\# M_k(x,t) = 2 \text{ and }U_{(x,t)}> q.
\end{cases}
\end{align*} 
It is important to observe that the point $m^q_k(x,t)$ is defined only when the set $M_k(x,t)$ is non-empty.

On the event $B(\mathbf{0})$, we define a path 
$\gamma_k = \gamma_k^{\mathbf{0}}(q)$ of length $k$ 
starting from the origin as 
\begin{align}
\begin{split}
\gamma_k(0) := \mathbf{0} \text{ and }
\gamma_k(j+1) := m^q_{k-j}(\gamma_{k}(j)), \text{ for }j=0,\ldots,k-1.
\end{split}
\end{align}
We first observe that conditional on the event $B(\mathbf{0})$,
 the set $M_k(\mathbf{0})$
is non-empty for all $k$ and hence the path $\gamma_k$ is well defined for all $k$.
The set of percolation points is denoted by ${\cal K}$.  
Set $(T_0, Y_0) :=  (0,\mathbf{0})$ and for $j \geq 1$ let 
\begin{align}
\label{def:RegenerationTimed2}
\begin{split}
T_j = T_j(q) &:= \inf\{k > T_{j-1} : \gamma_k(k) \in {\cal K}\}, \\
Y_j = Y_j(q) &:= \gamma_{T_j}(T_j) - \gamma_{T_{j-1}}(T_{j-1}).
\end{split}
\end{align}
It is not difficult to see that
$||Y_j||_1 \leq (T_j - T_{j-1})$ for all $j \geq 1$. Same argument as in Lemma 2.5 of \cite{BCDN13} gives us the following lemma, which 
shows that these steps are indeed regeneration steps.
We need to do the obvious modifications for defining appropriate filtration,
as here we are working with supercritical oriented bond percolation instead of site percolation.
For completeness we give the full details here:
\begin{lemma}
\label{lem:SingleRegIIDExpTail}
Conditional on the event $B(\mathbf{0})$,
the sequence $\{(T_j - T_{j-1},Y_j) : j \geq 1\}$ is i.i.d. and for all $n \geq 1$ we have 
\begin{equation}
\label{eq:SingleRegIIDExpTail}
\P_p(T_1 \geq n| B(\mathbf{0})) \leq C_1\exp{(-C_2 n)},
\end{equation}
where the constants $C_1,C_2 > 0$ depend only on $p$. 
\end{lemma}
\noindent {\bf Proof :} For $\x \in \Z^2$ and $\y\in V(\x)$, the  edge between these two neighbouring vertices  is denoted by $\langle \x , \y\rangle$. Let $\Z^2_+ := \{\x \in \Z^2 : \x(1) \wedge \x(2) \geq 0 \}$ denote the positive quadrant.
For $0 \leq n < m$, we define the filtration 
$${\cal G}^m_n := \sigma \bigl( \{(t(\langle \x , \y\rangle),U_\x) : \x \in \Z^2_+ , n \leq ||\x||_1 < m, \y \in V(\x)\} \bigr),$$
where $t(\langle \x , \y\rangle)$ is the passage time (random) 
attached to the edge $\langle \x , \y\rangle$.
We denote $\P_p(\cdot | B(\mathbf{0}))$ as $\tilde{\P}_p(\cdot )$.
Let $\sigma$ be a $\{{\cal G}^k_0 : k \geq 1\}$ stopping time. We first 
show that for any ${\cal G}^{\sigma}_0$ measurable $\Z^2_+$ valued random variable
$W$ with $||W||_1 = \sigma$, we have
\begin{equation}
\label{eq:FKGPercolating}
\tilde{\P}_p( B(W) | {\cal G}^{\sigma}_0) \geq \P_p(B(\mathbf{0})).
\end{equation}
For $n\geq 1$, let $S_n := \{\y \in \Z^2 : \y \in \Z^2_+, ||y||_1 = n\}$ and for any non-empty subset $S$ of $S_n$, the event $\{(0,0) \Rightarrow S\}$ represents that the set of vertices $\{\w : (0,0) \leadsto \w, ||\w||_1 = n \}$ equals exactly $S$. Then for each $n \geq 1$, the event $B(\mathbf{0})$ can be written as disjoint union of events
as follows:
\begin{equation}
\label{eq:B0disjtUni}
B(\mathbf{0}) = \bigcup_{S \subseteq S_n}\{ (0,0) \Rightarrow S \}\bigcap \bigl(\bigcup_{\y \in S}B(\y)\bigr).
\end{equation}
Clearly the event $\{ (0,0) \Rightarrow S \}$ is in ${\cal G}^{n}_0$.
Hence for any event $A$ in ${\cal G}^{\sigma}_0$ we obtain
\begin{align*}
& \P_p( B(W)\cap A \cap B(\mathbf{0}))\\
& = \sum_{n \in \N} \quad \sum_{\w \in \Z^2_+,||\w||_1 =n} \P_p(\{\sigma = n, W = \w \} \cap  B(\w) \cap  A\cap B(\mathbf{0}))\\
& = \sum_{n \in \N} \quad \sum_{\w \in \Z^2_+,||\w||_1 =n} \sum_{S \subseteq S_n}  \P_p \bigl(\{\sigma = n, W = \w \}  \cap  A \cap \{(0,0) \Rightarrow S \}\cap B(\w) \cap \bigl(\cup_{\y \in S}B(\y)\bigr)\bigr).
\end{align*}
The last step follows from (\ref{eq:B0disjtUni}).
Since the two sigma fields ${\cal G}^n_0$ and ${\cal G}^\infty_n$ are independent,
we further have
\begin{align*}
& = \sum_{n \in \N} \quad \sum_{\w \in \Z^2_+,||\w||_1 =n} \sum_{S \subseteq S_n}  \P_p \bigl(\{\sigma = n, W = \w \}  \cap  A \cap \{(0,0) \Rightarrow S \}\bigr)\P_p\bigl(B(\w) \cap \bigl(\cup_{ \y \in S}B(\y)\bigr)\bigr)\\
& \leq \sum_{n \in \N}  \quad \sum_{\w \in \Z^2_+,||\w||_1 =n} \sum_{S \subseteq S_n}  \P_p \bigl(\{\sigma = n, W = \w\}  \cap  A \cap \{(0,0) \Rightarrow S \}\bigr)\P_p(B(\w)) \P_p\bigl(\cup_{\y \in S}B(\y)\bigr)\\
& = \sum_{n \in \N}  \quad \sum_{\w \in \Z^2_+,||\w||_1 =n} \sum_{S \subseteq S_n}  \P_p \bigl(\{\sigma = n, W = \w\}  \cap  A \cap \{(0,0) \Rightarrow S \}\bigr) \P_p\bigl(\cup_{\y \in S}B(\y)\bigr)\P_p(B(\mathbf{0}))\\
& = \P_p( A \cap B(\mathbf{0}))\P_p( B(\mathbf{0})).
\end{align*}
In the second line we have used the FKG inequality. 
Since the event $A \in {\cal G}^\sigma_0$ is chosen arbitrarily, (\ref{eq:FKGPercolating}) follows.

As observed in \cite{BCDN13}, for $\sigma = \sigma_0 := 1$
and $W = \gamma_{\sigma_0}(\sigma_0)$ using (\ref{eq:FKGPercolating}) we have
$$
\tilde{\P}_p(T_1 = 1) = \tilde{\P}_p(T_1 = 1 | {\cal G}^1_0) = 
\tilde{\P}_p(\gamma_1(1) \in {\cal K} | {\cal G}^1_0) \geq \P_p(B(\mathbf{0})).
$$
When $\gamma_1(1)$ is not a percolation point, we wait for the local construction to discover it. On the event $E_1 := \{\gamma_1(1) \notin {\cal K}\}$,  the
random variable $l(\gamma_1(1))$ denoting the length of the longest oriented open path starting from $\gamma_1(1)$, is finite. While constructing $\gamma_{\sigma_1}$, where 
 $\sigma_1 := \sigma_0 + l(\gamma_1(1)) + 1$, we discover the fact that $\{\gamma_1(1) \notin {\cal K}\}$. Thus $\sigma_1$ is stopping time w.r.t. $\{{\cal G}^k_0 : k \geq 0\}$ and the event $E_1$ is in ${\cal G}^{\sigma_1}_0$. Hence (\ref{eq:FKGPercolating}) gives us that
$$
\mathbf{1}_{E_1}\tilde{\P}_p(T_1 = \sigma_1 | {\cal G}^{\sigma_1}_0) = 
\mathbf{1}_{E_1}\tilde{\P}_p(\gamma_{\sigma_1}(\sigma_1) \in {\cal K} | {\cal G}^{\sigma_1}_0) \geq \mathbf{1}_{E_1}\P_p(B(\mathbf{0})).
$$
Repeating the same argument recursively 
for $E_{k+1} := \{\gamma_{\sigma_k}(\sigma_k) \notin {\cal K}\}$ and 
$\sigma_{k+1} := \sigma_k + l(\gamma_{\sigma_k}(\sigma_k)) + 1$ we obtain
\begin{align*}
& \mathbf{1}_{E_{k+1}\cap E_{k}\cap \ldots \cap E_{1}}\tilde{\P}_p(T_1 = \sigma_{k+1} | {\cal G}^{\sigma_{k+1}}_0)\\
& =  \mathbf{1}_{E_{k+1}\cap E_{k}\cap \ldots \cap E_{1}}\tilde{\P}_p(\gamma_{\sigma_{k+1}}(\sigma_{k+1}) \in {\cal K} | {\cal G}^{\sigma_{k+1}}_0)\\
& \geq \mathbf{1}_{E_{k+1}\cap E_{k}\cap \ldots \cap E_{1}}\P_p(B(\mathbf{0})).
\end{align*}
This shows that the number of $\sigma_k$'s tested to find the value of $T_1$ is dominated  by a geometric random variable with success probability $\P_p(B(\mathbf{0}))$. Further
the random variable $\sigma_{k+1} - \sigma_{k}$
has exponential tail for all $k \geq 0$ (see \cite{D84}). 
Finally we observe that on the 
event $E_{k+1}$ the distribution of $\sigma_{k+1} - \sigma_{k}$ as well as $\P(B(\mathbf{0}))$ do not depend on the parameter $q$. This proves (\ref{eq:SingleRegIIDExpTail}).

Finally the proof that the sequence $\{(T_j - T_{j-1}, Y_j) : j \geq 1\}$ is i.i.d., follows from
the same arguments as that of Lemma 2.5 of \cite{BCDN13}.
This basically uses the fact that,
 for any $k \geq 1$ at the $k$-th renewal step $T_k$, the 
only information we have about the future is $\gamma(T_k) \in {\cal K}$.
\qed

From the arguments of the earlier lemma, it follows that the distribution of $T_1$ does not depend on the choice of $q \in [0,1]$.
In general the construction of the path $\gamma_k$ is not time consistent in the sense that for $l \leq j < k$ we may have $\gamma_k(l) \neq \gamma_j(l)$.
But as observed in Lemma 2.1 of \cite{BCDN13}, for some $k \geq 1$ if we have $\gamma_k(k) \in {\cal K}$, then we have $\gamma_m(l) = \gamma_k(l)$ for all $m \geq k$ and $l \leq k$.
This observation together with Lemma \ref{lem:SingleRegIIDExpTail} shows that conditional on the event $B(\mathbf{0})$,  for all $j \in \mathbb{N}$ the limit, $\lim_{k \to \infty} \gamma_k(j)$ exists almost surely. Further by construction, $\{\gamma(j)= \lim_{k \to \infty} \gamma_k(j) : j \geq 0\}$ gives an \textit{infinite oriented open} path starting from the origin. 
In what follows, we denote the conditional probability measure $\mathbb{P}_p(\cdot | B(\mathbf{0}))$ as $\tilde{\mathbb{P}}_p(\cdot )$. Now we prove the following lemma.

\begin{lemma}
Conditional on the event $B(\mathbf{0})$, the infinite path $\gamma$ constructed before has a deterministic direction almost surely.
\end{lemma}
\begin{proof}
Conditional on the event $B(\mathbf{0})$, for the infinite oriented open path $\gamma$ we have that $||\gamma(T_j)||_1  =  T_j$ for any $j \geq 1$. Hence we obtain that $\tilde{{\mathbb P}}_p(\cdot)$ almost surely
\begin{align}
\label{eq:DirLimRegSubseq}
\lim_{j \to \infty} \frac{\gamma(T_j)}{||\gamma(T_j)||_1} = \lim_{j \to \infty} \frac{\gamma(T_j)}{T_j}
=  \lim_{j \to \infty} \frac{\gamma(T_j)/j}{T_j/j}
=  \lim_{j \to \infty} \frac{(\sum_{i=1}^j Y_i)/j}{(\sum_{i=1}^j (T_i - T_{i-1}))/j}
= \frac{\mathbb{E}(Y_1)}{\mathbb{E}(T_1)}.
\end{align}
The last step follows from Lemma \ref{lem:SingleRegIIDExpTail}.
What remains to show is that, $$\lim_{j \to \infty} \gamma(j)/||\gamma(j)||_1 = \lim_{j \to \infty} \gamma(T_j)/||\gamma(T_j)||_1.$$

Let $\theta = \theta(q)$ be the argument of the (deterministic) unit vector ${\mathbb E}(Y_1)/{\mathbb E}(T_1)$. Fix $m \in \mathbb{N}$ and let $B_m$ denote the event that
\begin{align*}
B_m := \{& \text{there exists a subsequence }\{\gamma(j_k)/||\gamma(j_k)||_1 : k \in \mathbb{N}\}\text{ such that }\\
& \lim_{k \to \infty} \gamma(j_k)/||\gamma(j_k)||_1 \text{ exists and }\text{arg}\bigl(\lim_{k \to \infty} \gamma(j_k)/||\gamma(j_k)||_1 \bigr) = \theta^\prime\\&\text{ with }|\theta^\prime - \theta| \geq \pi/m\}.
\end{align*}
Since $m \in \mathbb{N}$ is chosen arbitrarily, it suffices to show that $\tilde{{\mathbb P}}_p(B_m) =0$.
We consider the case that $\theta, \theta^\prime \in (0, \pi/4)$ with $\theta^\prime > \theta$.
For the other cases, the argument is similar.

We consider two disjoint regions in $\mathbb{R}^2$ defined as
\begin{align*}
{\cal R}^- &:= \{(x,t) : x > 0, t \in (0, \tan(\theta + \pi/4m)x) \}
\text{ and }\\
{\cal R}^+ &:= \{(x,t) : x > 0, t > \tan(\theta^\prime - \pi/4m)x\}.
\end{align*}
From (\ref{eq:DirLimRegSubseq}) it follows that there exists $j_0 = j_0(\omega)$ such that
$\gamma(T_j) \in {\cal R}^-$ for all $j \geq j_0$. Our assumption on the subsequence $\{\gamma(j_k) : k \in \mathbb{N}\}$ gives us that there exists $k_0 = k_0(\omega)$ as well, with
$\gamma(j_k) \in {\cal R}^+$ for all $k \geq k_0$.

For $n \in \mathbb{N}$ let $j^n \geq 1$ be the random index such that $j^n := \min\{j \geq 0 : T_j \geq n\}$. We define the event  $A_n(m) := 
\{ T_{j^n + 1} - T_{j^n} \geq  n \tan(\frac{\pi}{4m})\}$. Recall that $(\theta^\prime - \theta ) \geq \pi/m$. Since $||\gamma(T_{j^n } + l) - \gamma(T_{j^n })||_1 \leq T_{j^n + 1} - T_{j^n}$
for any $T_{j^n} \leq l \leq T_{j^n + 1} $, we observe that 
$$\{\gamma(T_{j^n + 1}), \gamma(T_{j^n }) \in {\cal R}^-\}\cap \{\gamma(l) \in {\cal R}^+ \text{ for some }T_{j^n} < l < T_{j^n + 1}\} \subseteq A_n(m).$$
Further, existence of both the random integers $j_0$ and $k_0$ shows that 
$$ B_m \subseteq \limsup_{n \to \infty} A_n(m).$$
On the other hand, (\ref{eq:SingleRegIIDExpTail}) gives us that $\sum_{n = 1}^\infty\tilde{{\mathbb P}}_p(A_n(m)) < \infty$.
Hence by the first Borel-Cantelli lemma we have that $\tilde{{\mathbb P}}_p(B_m) = 0$.
Since $m \in \mathbb{N}$ is chosen arbitrarily, this completes the proof.
\end{proof}

Now we proceed to the proof of Proposition \ref{prop:DirectedPathforOPd2}.

\begin{proof}(\textit{of Proposition \ref{prop:DirectedPathforOPd2}})
 The previous lemma gives us  that for each $q \in [0,1]$, conditional on the event $B(\mathbf{0})$, the infinite path $\gamma = \gamma(q)$ constructed above, has fixed asymptotic direction $\theta$ which depends on $q$. 
From the construction of $\gamma$ it further follows that the choice of $q = 1$ ($q = 0$) gives the \textit{rightmost (leftmost)} infinite oriented open path starting from the origin and hence $\theta(1) = \theta^-_p $ and $\theta(0) = \theta^+_p $. By standard coupling arguments, 
it follows that $\theta(q)$, the asymptotic direction of the `$q$-path' $\gamma$,
is a decreasing function of $q$. Hence to prove Proposition \ref{prop:DirectedPathforOPd2}, it suffices to show that the mapping $\theta : [0,1] \mapsto [\theta^-_p,\theta^+_p ]$ is continuous. 
Since distribution of the random variable $T_1(q)$ does not depend upon $q$, clearly
the function $q \mapsto {\mathbb E}(T_1(q))$ is continuous. It remains to show
that the function $q \mapsto {\mathbb E}(Y_1(q))$ is continuous.
 This part of the proof is motivated from Remark 2.6 of \cite{BCDN13}. Let $\mathbb{Z}^{2}_{+} := \{(y,s) \in \mathbb{Z}^2 : y,  s \geq 0 \}$.
We observe that for any $n \geq 1$ and for $(y,s) \in \mathbb{Z}^{2}_{+}$ with $y+s = n$, we have
\begin{align*}
q \mapsto \tilde{{\mathbb P}}_p( \gamma(T_1) = (y,s)) 
= & \tilde{{\mathbb P}}_p( T_1 = n, Y_1 = (y,s))\\
 = & \tilde{{\mathbb P}}_p(\gamma_n(n) = (y,s),B(y,s), j + l(\gamma_j(j)) < n \text{ for  }1\leq j < n)  \\
= & \tilde{{\mathbb P}}_p(\gamma_n(n) = (y,s),j + l(\gamma_j(j)) < n \text{ for  }1\leq j < n)\tilde{{\mathbb P}}_p(B(y,s))
\end{align*}
is continuous on $[0,1]$. The two events, $\{\gamma_n(n) = (y,s),j + l(\gamma_j(j)) < n \text{ for }1\leq j < n\}$ and $B(y,s)$, depend on disjoint set of edges and hence they are independent.
We have also used the fact that for any $(y,s) \in \mathbb{Z}^{2}_{+}$ with $y+s = n$, the probability ${\mathbb P}_p(\gamma_n(n) = (y,s),j + l(\gamma_j(j)) < n \text{ for  }1\leq j < n)$ depends only on finitely many passage time random variables and uniform random variables. On the other hand, the probability $\tilde{{\mathbb P}}_p(B(y,s))$ does not depend on $q$. 

In order to compute ${\mathbb E}(Y_1)$ we observe that
\begin{align*}
{\mathbb E}(Y_1) = \sum_{n=1}^{\infty}~ \sum_{(y,s)\in \mathbb{Z}^{2}_{+}, y+s = n }(y,s)\tilde{{\mathbb P}}_p(T_1 = n, Y_1 = (y,s)). 
\end{align*} 
Because of (\ref{eq:SingleRegIIDExpTail}) for any $\epsilon > 0$ we can choose $n_0$ uniformly over $q\in [0,1]$ such that 
\begin{align*}
& \sum_{n= n_0}^{\infty}~ \sum_{(y,s)\in \mathbb{Z}^{2}_{+}, y+s = n }
||(y,s)||_1\tilde{{\mathbb P}}_p(T_1 = n, Y_1 = (y,s)) \\
& =  \sum_{n= n_0}^{\infty}~ \sum_{(y,s)\in \mathbb{Z}^{2}_{+}, y+s = n }
n \tilde{{\mathbb P}}_p(T_1 = n, Y_1 = (y,s)) < \epsilon.
\end{align*}
Hence continuity of the function $q \mapsto {\mathbb E}(Y_1)$ follows. 
\end{proof}
Now we fix $q \in [0,1]$ and 
state our next proposition. 
  
\begin{proposition}
\label{prop:CoalesceQpathd2}
For $\x, \y\in \mathbb{Z}^{2}$ with $(\x(1) + \x(2)) = (\y(1) + \y(2))$,
conditional on the event $B(\mathbf{x})\cap B(\mathbf{y})$, the two paths $\gamma^{\x}$ and $\gamma^{\y}$ coalesce almost surely.
\end{proposition}
\begin{proof}
 We first define regeneration times
for the process $\{(\gamma^{\x}(j),\gamma^{\y}(j)) : j \geq 0\}$ as follows.
Set $\tau_0 = \tau_0(\x, \y) = 0$ and for $j \geq 1$ we define
\begin{align}
\label{def:JtRegQ}
\tau_j = \tau_j(\x, \y) := \inf\{n > \tau_{j-1}(\x, \y) : \gamma^{\x}(n), \gamma^{\y}(n) \in {\cal K}\}.
\end{align}
Same arguments as in the proof of Lemma 3.1 of \cite{BCDN13} give us that for all $j ,n \geq 1$,
\begin{equation}
\label{eq:JtRegExpTail}
{\mathbb P}(\tau_j(\x,\y) - \tau_{j-1}(\x,\y) \geq n) \leq C_3\exp{(-C_4n)},
\end{equation}
where the constants $C_3,  C_4 >0$ depend only on $p$. The brief heuristics are as follows :
if it is a marginal regeneration for the first path, only then examine for the regeneration for the second path  and an
 application of FKG inequaity gives us that the probability of such an event is at least ${\mathbb P}_p(B(\mathbf{0}))$, which is strictly positive for $p > p^\rightarrow_c$.

For each $j\geq 1$, at the $j$-th joint regeneration step $\tau_j$, the only 
information that we have about the future is that both the points, $\gamma^\x(\tau_j)$ and 
$\gamma^\y(\tau_j)$, are in ${\cal K}$. Similar argument as in \cite{BCDN13} give us that 
the process $\bigl\{\bigl(\gamma^{\x}(\tau_j),\gamma^{\y}(\tau_j)\bigl): j \geq 0 \bigl\}$
is Markov. By construction for all $j \geq 0$ we have
 $$
 (\gamma^{\x}(\tau_j)(1) + \gamma^{\x}(\tau_j)(2)) = (\gamma^{\y}(\tau_j)(1) + \gamma^{\y}(\tau_j)(2)) = (\x(1) + \x(2)) + \tau_j.$$
  From translation invariance of our model and from the above observation, it follows that the process 
\begin{equation}
\label{def:Z_l}
\{Z_j = Z_j ( \x, \y) := ||\gamma^{\x}(\tau_j) - \gamma^{\y}(\tau_j)||_1 : j \geq 0\}
\end{equation}
is a non-negative Markov chain with $ 0 $ as the only absorbing state. 
It suffices to show that the Markov chain $\{Z_j : j \geq 0\}$ will be absorbed at $0$
eventually.

Let $m_0 \in 2\mathbb{N}$ be a constant which will be specified later. 
For $\x, \y$ chosen as above, let $\nu = \nu (\x, \y) := \inf\{j \geq 1 : Z_j (\x,\y)  \leq  m_0 \}$ denote the first time that the process enters $[0,m_0]$.
For all $m \leq m_0$ there exists $p_0 = p_0(m_0) > 0$
such that $${\mathbb P}(Z_{j + \frac{m_0}{2}} = 0 | Z_{j} = m) \geq p_0.$$ Hence to prove Proposition \ref{prop:CoalesceQpathd2} it suffices to show that the hitting time $\nu(\x,\y)$ is almost surely finite.  The following lemma proves this fact and hence completes the proof of Proposition \ref{prop:CoalesceQpathd2}.
\end{proof}

\begin{lemma}
\label{lem:QCoalesceTimeFinite}
For $\x,\y$ chosen as above, conditional on the event $B(\x)\cap B(\y)$, the hitting time $\nu = \nu(\x,\y)$ is finite almost surely .
\end{lemma}
\begin{proof} Let $\mathbf{1}_E$ denote indicator function of the event $E$.
 Define a non-negative process  $
 \{L_j := \mathbf{1}_{\{\nu(\x,\y) > j\}}f(Z_j(\x,\y)) : j \geq 1\}$
 where the mapping $f : (0,\infty) \mapsto \mathbb{R}$ is given by $f(x) := \log(x)$.
We then prove that the process $\{L_j : j \geq 1\}$ is a supermartingale. Being a non-negative supermartingale  the process $\{L_j : j \geq 1\}$
must have an almost sure limit. On the other hand, the function $\log(\cdot)$ is bijective, implying that the process $\{L_j : j \geq 1\}$ is Markov with the only absorbing state $0$. Hence we have $L_j \to 0$ as $j \to \infty$ almost surely.
Since it is a discrete valued process, $\nu$
is finite almost surely.

It suffices to show that the process $\{L_j : j \geq 0\}$
is a supermartingale. We observe that $Z_{j + 1} - Z_j \leq 2(\tau_{j+1} - \tau_j)$ for all $j \geq 0$. Hence from (\ref{eq:JtRegExpTail}) we have
${\mathbb E}(L_{j+1} |  L_{j}) < \infty$.

Let us suppose $ Z_j = m $ for some $m > m_0$. 
Next we consider two disjoint $||\quad||_1$ triangles, $R_1$ and $R_2$ centered around
the points $\gamma^{\x}(\tau_j)$ and $\gamma^{\y}(\tau_j)$ respectively, given by 
\begin{align*}
R_1 & := \{\z : (\z - \gamma^{\x}(\tau_j)) \in \mathbb{Z}^2_+ , ||\z - \gamma^{\x}(\tau_j)||_1 < m/2  \}\text{ and }\\
R_2 &:= \{\z : (\z - \gamma^{\y}(\tau_j)) \in \mathbb{Z}^2_+ , ||\z - \gamma^{\y}(\tau_j)||_1 < m/2  \}.
\end{align*}
We interchange the realizations of the passage time random variables and uniform random variables associated to edges and vertices inside $R_1$ by those in $ R_2$ and vice versa.
Restricting our attention to the event $ \{ \tau_{j+1} - \tau_{j} < m/2\}$, for the 
paths starting from  $ \gamma^{\x}(\tau_j)$ and $ \gamma^{\y}(\tau_j)$, 
constructed using the original realizations,  and the interchanged realizations,
we observe that the number 
of steps required for the next regeneration remain exactly the same.
 Further, the increment of the 
first path (starting from $\gamma^{\x}(\tau_j)$) at the next regeneration step
 (using the new realizations) becomes the increment of the second path 
starting from $\gamma^{\y}(\tau_j)$ (using the original or old realizations)  
and vice versa. 

Taking $ I :=  Z_{j+1} - Z_{j} $, the above discussions can be summarized as
\begin{equation}
\label{eq:Isymmetry}
\mathbf{1}_{\{(\tau_{j+1} - \tau_{j}) < m/2\}} I \;| \{ Z_{j} = m \}
\stackrel{d}{=}  -  \mathbf{1}_{\{(\tau_{j+1} - \tau_{j}) < m/2\}} I \;| \{ Z_{j} = m\}.
\end{equation}

Now for $Z_j = m > m_0$ we obtain
\begin{align}
\label{CalcLJSuperMart}
 \lefteqn{  \mathbb{E}[(L_{j+1} - L_{j})| L_{j} = f(m) ] } \nonumber \\
& \leq    \mathbb{E}[f(Z_{j+1}) - f(Z_{j})| Z_{j} = m]\nonumber\\
 & =  \mathbb{E} \bigl [\bigl ((f(Z_{j+1}) - f(Z_{j}))\mathbf{1}_{\{\tau_{j+1} - \tau_{j} < m/2\}}\bigr) + \bigl( (f(Z_{j+1}) - f(Z_{j}))\mathbf{1}_{\{\tau_{j+1} - \tau_{j} \geq m/2\}} \bigr) \bigl | Z_{j} = m \bigr ] \nonumber\\
& \leq  \mathbb{E}[(f(m + I) - f(m))\mathbf{1}_{\{\tau_{j+1} - \tau_{j} < m/2\}} \;| 
 Z_{j} = m ] +
2\mathbb{E}[(\tau_{j+1} - \tau_{j})\mathbf{1}_{\{\tau_{j+1} - \tau_{j} \geq m/2\}}| 
 Z_{j} = m ].
\end{align}
Before we calculate the terms we note that, for any $ l > 0 $,   
\begin{align*}
& \mathbb{E}[ (\tau_{j+1} - \tau_{j})^l \mathbf{1}_{\{\tau_{j+1} - \tau_{j} \geq m/2\}}| 
 Z_{j} = m ] \\
\leq & \mathbb{E}[ (\tau_{j+1} - \tau_{j})^{2l} | 
 Z_{j} = m ]^{1/2} \mathbb{P} [ \tau_{j+1} - \tau_{j} \geq m/2 |  Z_{j} = m ]^{1/2} \\
\leq &  C_3 \beta_m(l) \exp ( -  C_4 m/4 )    
\end{align*}
where  $ C_3$ and $C_4$ are as in (\ref{eq:JtRegExpTail}). 
In order to study the behaviour of the sequence  $\beta_m(l)$ as $m\to \infty$, 
we need to introduce few notation.

Let $\{T^{(1)}_n : n \in \N\}$ and $\{T^{(2)}_n : n \in \N\}$  be two independent families of 
i.i.d. copies of $T_1(\mathbf{0})$.
For $m \in \mathbb{N}$, set $S^{(i)}_m := \sum_{j=1}^{m} T^{(i)}_j $ for $i = 1,2$ and define
$Y := \inf\{S^{(1)}_m : m \geq 1, S^{(1)}_m = S^{(2)}_l\text{ for some }l\geq 1\}$.

It is konwn that $Y$ has finite moments of all orders (see Subsection 3.2 of \cite{RSS15} for details). For $\beta_m(l)$, we note that the probability $\mathbb{P} [ \tau_{j+1} - \tau_{j} \geq m/2 |  Z_{j} = m ] \to 1$ as $m\to \infty$. Hence for large $m$, in between $\tau_j$ and $\tau_{j+1}$, with high probability both the paths explore disjoint regions.
This gives us that $\beta_m(l) \to \mathbb{E}[Y^{2l}]^{1/2} > 0$ as $m\to \infty$.

Now for the first term, we have $ f^{(2)} (x) = - 1/x^2,  
f^{(3)} (x) = + 2/x^3, $ and $ f^{(4)} (x) = - 6/x^4 < 0 $ for all 
$ x > 0 $. Thus, using Taylor's expansion, we have
\begin{equation*}
f(m+I) - f(m) \leq I f^{\prime} (m) - \frac{ 1}{ m^2 } I^2 + \frac{ 2}{ m^3 } I^3  .
\end{equation*} 
Therefore, 
\begin{align*}
\lefteqn{ \mathbb{E}[(f(m + I) - f(m))\mathbf{1}_{\{\tau_{j+1} - \tau_{j} < m/2\}} \;| 
 Z_{j} = m ] }  \\  
&   \leq f^{\prime} (m) \mathbb{E}[ I \mathbf{1}_{\{\tau_{j+1} - \tau_{j} < m/2\}} \;|  Z_{j} = m ]   
- \frac{ 1}{ m^2 } \mathbb{E}[ I^2 \mathbf{1}_{\{\tau_{j+1} - \tau_{j} < m/2\}} \;| Z_{j} = m  ]  \\
& \qquad +    \frac{2}{ m^3} \mathbb{E}[ I^3 \mathbf{1}_{\{\tau_{j+1} - \tau_{j} < m/2\}} \;| Z_{j} = m  ] \\
& = - \frac{ 1}{ m^2 } \mathbb{E}[ I^2 \mathbf{1}_{\{\tau_{j+1} - \tau_{j} < m/2\}} \;| Z_{j} = m  ]  
\end{align*} 
using (\ref{eq:Isymmetry}).
 Further, it is easy to observe that there exists $\alpha > 0$ such that by creating suitable configurations, 
we have $ \mathbb{P} ( | I | \mathbf{1}_{\{\tau_{j+1} - \tau_{j} < m/2\}} \geq 1  \;|  Z_{j} = m ) \geq \alpha > 0 $ 
for all $ m > m_0 $. Therefore, $\mathbb{E}[ I^2 \mathbf{1}_{\{\tau_{j+1} - \tau_{j} < m/2\}} \;| Z_{j} = m]\geq \alpha $
for any $ m \geq m_0 $. Thus, we have, 
\begin{align*}
  \mathbb{E}[(L_{j+1} - L_{j})| L_{j} = f(m) ]
 \leq - \frac{ \alpha }{ m^2 }  + 2  C_3 \exp ( - C_4 m / 4 ).
\end{align*}
So, for a suitable choice of $ m_0 \in 2\mathbb{N}$ and for all $ m > m_0$, we have 
 $ \mathbb{E}[(L_{j+1} - L_{j})| L_{j} = f(m) ] \leq 0 $. This completes the proof.
\end{proof}

In order to prove Theorem \ref{thm:DirectedGeodesicsford2}, we first of all observe that
since we are working with passage time distribution satisfying (\ref{def:DurettLiggettClass}), any oriented open path (finite or infinite) must be a geodesic. Because of Proposition \ref{prop:DirectedPathforOPd2}, for any $\theta \in [\theta^-_p, \theta^+_p]$, 
on the event $B(\mathbf{0})$ 
there exists $q \in [0,1]$ such that the infinite oriented `$q$' path $\gamma$ (which is a geodesic as well) has asymptotic direction $\theta$. On the complement event we need a `sandwiching' argument. In order to do that we describe 
\textit{bi-directional percolation point} which was used in \cite{WZ08}. 
We say that there is an \textit{anti-oriented} open path from 
$\y$ to $\x$ if $\x \leadsto \y$.
The set of points with infinite anti-oriented open paths is denoted by ${\cal K}^{\text{anti}}$. 
\begin{definition}
\label{def:BiPercPoint}
Any point $\x \in \mathbb{Z}^2$ with $\x \in {\cal K}\cap {\cal K}^{\text{anti}}$ is termed a \textit{bi-directional percolation point}.
\end{definition}
As observed in \cite{WZ08}, for $p > p^{\rightarrow}_{c}$ we have 
$$
{\mathbb P}_p(\mathbf{0}\text{ is a bi-directional percolation point}) = {\mathbb P}_p(B(\mathbf{0}))^2 > 0.
$$
While proving Theorem \ref{thm:DirectedGeodesicsford2}, we show that there are bi-directional percolation points ${\mathbb P}_p$ almost surely for $p>p^\rightarrow_c$.
Before proceeding further we comment that for any bi-directional percolation point,
the concatenation of infinite oriented open path and infinite anti-oriented open path
gives a bi-infinite geodesic path. Hence existence of bi-directional percolation point proves Theorem \ref{thm:BiInfGeoesic}.

\begin{proof}(\textit{of Theorem \ref{thm:DirectedGeodesicsford2}})
Note that on the event $B(\mathbf{0})$, the proof of Theorem \ref{thm:DirectedGeodesicsford2} follows from Proposition 
\ref{prop:DirectedPathforOPd2}. On the event $B(\mathbf{0})^c$, from the origin we consider the nearest right and nearest left bi-directional percolation points on the line $ y = -x$.
 More formally, let
\begin{align*}
j_r & = j_r(\omega) := \min\{j > 0 : (j,-j) \in  \mathbb{Z}^2 \cap {\cal K}\cap {\cal K}^{\text{anti}}\} \text{ and }\\
j_l & = j_l(\omega) := \max\{j < 0 : (j,-j) \in \mathbb{Z}^2 \cap {\cal K}\cap {\cal K}^{\text{anti}}\}.
\end{align*}
It is not difficult to see that both the random variables, $j_r$ and $j_l$, are almost surely finite. Here we present the heuristics briefly. For $(y,s) \in \mathbb{Z}^2$, let $ l_{\text{anti}}(y,s)$ denote the length of the longest anti-oriented open path starting from $(y,s)$. On the event $B(\mathbf{0})^c$, set $l_0 := l(\mathbf{0})\wedge l_{\text{anti}}(\mathbf{0}) + 2$. Examining open (closed) status of the lattice paths,
consisting of the vertices in the set $V_0 := \{(y,s): \text{ either }(y,s)\in \mathbb{Z}^2_+, 
\text{ or }(-y,-s)\in \mathbb{Z}^2_+ , ||(y,s)||_1 \leq l_0\}$, we discover that $(0,0) \notin {\cal K}\cap {\cal K}^{\text{anti}}$.
Next we examine both the points $(l_0, -l_0)$ and $(-l_0, l_0)$. Whether the point $(l_0, -l_0)$ is a bi-directional percolation point or not depends on the edges consisting of the vertices
in the set $$V^{r}_1 := \{(y,s) :\text{ either }((y,s) - (l_0,-l_0)) \in \mathbb{Z}^2_+, 
\text{ or }((l_0,-l_0) - (y,s)) \in \mathbb{Z}^2_+ \}.$$
 Since both the sets are disjoint, we have
 ${\mathbb P}_p\bigl((l_0,-l_0) \in {\cal K}\cap {\cal K}^{\text{anti}}\bigr) = {\mathbb P}_p(B(\mathbf{0}))^2$. On the other hand, on the event   $\{(l_0,-l_0) 
 \notin {\cal K}\cap {\cal K}^{\text{anti}}\}$, it is enough to examine the lattice paths
consisting of the vertices in the set $V^{r}_1 \cap \{(y,s) : ||(y,s) - (l_0,-l_0)||_1
\leq  l^r_1\}$ where   $l^r_1 := l(l_0,-l_0)\wedge l_{\text{anti}}(l_0,-l_0) + 2$
 and we move on to examine the next point $((l_0 + l^r_1), - (l_0 + l^r_1))$.
 Our choice ensures that the corresponding sets are disjoint and hence again we have ${\mathbb P}_p\bigl(((l_0 + l^r_1), - (l_0 + l^r_1))\in {\cal K}\cap {\cal K}^{\text{anti}}\bigr) = {\mathbb P}_p(B(\mathbf{0}))^2$. 
Similar argument holds for the point $(-l_0, l_0)$ also. Hence a geometric argument shows that
both the random integers, $j_r$ and $j_l$, are finite almost surely. 

Fix $q \in [0,1]$ and  consider the infinite oriented open $q$ paths $\gamma^{(j_r, - j_r)}$ and $\gamma^{(-j_l,j_l)}$ starting from the points $(j_r, - j_r)$ and $(-j_l,j_l)$ respectively. Let 
$$
n_0 = n_0(\omega) := \min\{n \geq 1 :\gamma^{(j_r, - j_r)}(n) = \gamma^{(-j_l,j_l)}(n)\}.
$$ Because of Proposition \ref{prop:CoalesceQpathd2}, $n_0$ is finite almost surely.
The infinite anti-oriented open $q$ paths starting from the points $(j_r, - j_r)$ and 
$(-j_l,j_l)$ are denoted by $\gamma^{(j_r, - j_r),\text{anti}}$ and $\gamma^{(-j_l,j_l),\text{anti}}$ respectively. Set $n_0^{\text{anti}} := \min\{n : \gamma^{(j_r, - j_r),\text{anti}}(n)=\gamma^{(-j_l,j_l),\text{anti}}(n)\}$, which is finite almost surely.

This ensures that the origin is enclosed by a pair of coalescing oriented geodesics and a pair of coalescing anti-oriented geodesics.
Let 
\begin{align*}
\Delta := \{&(y,s) \in \mathbb{Z}^2 : \gamma^{(-j_l,j_l)}(s) \leq y \leq \gamma^{(j_r, - j_r)}(s)\text{ for }0 \leq s \leq n_0\text{ or }\\
&\gamma^{(-j_l,j_l),\text{anti}}(s) \leq y \leq \gamma^{(j_r, - j_r),\text{anti}}(s)\text{ for }  0 \leq s \leq n_0^{\text{anti}} \}
\end{align*}
 denote the lattice points in the enclosed region. 
Next we choose a finite geodesic path between the origin and the point $\gamma^{(j_r,-j_r)}(n_0)$ using the lattice paths consisting of vertices from the set $\Delta$ only. We note that such a choice
is always possible. If not, i.e., if every geodesic path from $(0,0)$ to $\gamma^{(j_r,-j_r)}(n_0)$ contains a point outside the set $\Delta$, then this contradicts the fact that the bi-infinite open paths passing through the points $(j_r,-j_r)$ and $(-j_l,j_l)$, obtained by concatenations of the 
oriented $q$ paths and anti-oriented $q$ paths, are geodesics.
Similar argument shows that the concatenation 
of this finite geodesic with the oriented infinite open $q$ path starting from the point
$\gamma^{(j_r,0)}(n_0)$ onward gives an infinite geodesic. Since modifications are done on finitely many edges only, this newly constructed geodesic path will have the same asymptotic direction. 
This completes the proof of  Theorem \ref{thm:DirectedGeodesicsford2}. 
\end{proof}

\begin{proof}\textit{( of Theorem \ref{thm:CoalGeoesic})}
We first prove it for For $ \x, \y \in \mathbb{Z}^2$ chosen as above.
Fix $\theta \in [\theta^-_p, \theta^+_p]$ and choose $q \in [0,1]$ such that conditional on the event $B(\mathbf{0})$, the $q$ infinite oriented open path $\gamma$ almost surely has asymptotic direction $\theta$. For $ \x, \y \in \mathbb{Z}^2$ chosen as above with $\x(1) > \y(1)$, 
we consider the nearest bi-directional percolation points 
towards right of $\x$ and towards left of $\y$. These two percolation points are denoted by 
$\x^r$ and $\y^l$ respectively. Proposition \ref{prop:CoalesceQpathd2} gives us that
the infinite oriented open $q$ paths starting from these points 
coalesce almost surely. Let $n_1= n_1(\omega) := \min\{ n \geq 1 : \gamma^{\x^r}(n) = \gamma^{\y^l}(n)\}$. Further, the proof of Theorem \ref{thm:DirectedGeodesicsford2} shows that a finite geodesic between $\x$ and $\gamma^{\x^r}(n_1)$, which does not cross the bi-infinite $q$ paths passing through $\x^r$ and $\y^l$, concatenated with the infinite oriented $q$ path starting from $\gamma^{\x^r}(n_1)$ gives an infinite geodesic starting from $\x$ with asymptotic direction $\theta$. Same argument holds for $\y$ also. This proves Theorem \ref{thm:CoalGeoesic} for the specific choice of $ \x, \y $.

Next, we show that it is enough to prove Theorem \ref{thm:DirectedGeodesicsford2}
for the given choice of $\x, \y$.
Indeed, for general $\x, \y \in \mathbb{Z}^2$ with $(\x(1) + \x(2)) > (\y(1) + \y(2))$, the above argument gives,
\begin{align*}
 {\mathbb P}_p \bigl[ \bigcap_{ \w \in \mathbb{Z}^2, \x(1) + \x(2) = \w(1) + \w (2) } \{ \text{the paths } \gamma^{\x} \text{ and }\gamma^{\w} \text{  coincide eventually} \} \bigr] = 1; \\
 {\mathbb P}_p  \bigl[ \bigcap_{ \w^\prime \in \mathbb{Z}^2, \w^\prime (1) + \w^\prime (2) = \y(1) + \y (2) }  \bigl\{ \text{the paths } \gamma^{\y} \text{ and }
 \gamma^{\w^\prime} \text{  coincide eventually} \} \bigr] = 1.
\end{align*}
Since, $\gamma^\y((\x(1) + \x(2)) - (\y(1) + \y(2))) \in \mathbb{Z}^2$ with 
$$\bigl(\gamma^\y((\x(1) + \x(2)) - (\y(1) + \y(2)))(1) + \gamma^\y((\x(1) + \x(2)) - (\y(1) + \y(2)))(2)\bigr ) = (\x(1) + \x(2)),$$ the paths $ \gamma^{\x} $ and $ \gamma^{\y}$ meet.
This completes the proof of Theorem \ref{thm:CoalGeoesic}.
\end{proof}




\end{document}